\documentclass[12pt]{article}
\usepackage[top=2.5cm,bottom=2.5cm,left=2.9cm,right=2.9cm]{geometry}
\usepackage{amssymb}
\usepackage{amsmath,amsthm}
\usepackage[latin1]{inputenc}
\usepackage{graphicx}
\usepackage{hyperref}
\usepackage{enumerate}
\usepackage{tikz}
\usepackage{tkz-graph}
\usepackage{mathrsfs}
\usepackage{verbatim}
\usepackage{upgreek}
 \usepackage{mathptmx}
\usepackage{colortbl}
\usepackage{upgreek}
\usepackage{mathptmx}

\usepackage{cancel}
\usepackage{colortbl}
\usepackage{tabularray}

\usepackage{eucal}

\hypersetup{colorlinks=true}

\hypersetup{colorlinks=true, linkcolor=blue, citecolor=blue,urlcolor=blue}

\newtheorem{remark}{Remark}[section]

\newtheorem{lemma}[remark]{Lemma}
\newtheorem{theorem}[remark]{Theorem}
\newtheorem{proposition}[remark]{Proposition}

\usepackage{exscale}

\newcommand{\ml}{\left[\hspace*{-0.05cm}\left[}
\newcommand{\mr}{\right]\hspace*{-0.05cm}\right]}
\newcommand{\mode}[1]{\ml #1\mr}

\begin{document}

\title{Equidistant dimension of Cartesian product graphs}

\author{Adri\`a Gispert-Fern\'andez $^{a,}$\thanks{Email: \texttt{adria.gispert@estudiants.urv.cat}}
\and
Juan A. Rodr\'iguez-Vel\'azquez $^{a,}$\thanks{Email: \texttt{juanalberto.rodriguez@urv.cat}}
\and
Ismael G. Yero $^{b,}$\thanks{Email: \texttt{ismael.gonzalez@uca.es}}
}

\maketitle

\begin{center}
$^a$ Departament d'Enginyeria Inform\`atica i Matem\`atiques, Universitat Rovira i Virgili, Spain \\
	\medskip

	$^b$ Departamento de Matem\'aticas, Universidad de C\'adiz, Algeciras Campus, Spain
    \\
	\medskip
\end{center}
		
\begin{abstract}
Given a connected graph $G$, the equidistant dimension of $G$ represents the cardinality of the smallest set of vertices $S$ of $G$ such that for any two vertices $x,y\notin S$ there is at least one vertex in $S$ equidistant to both $x,y$ in terms of distances. In this article, we compute the equidistant dimension of some Cartesian product graphs including two-dimensional Hamming graphs, some hypercubes, prisms of cycle, and squared grid graphs.
\end{abstract}


\noindent
AMS Subj.\ Class.\ (2020): 05C12, 05C76

\section{Introduction}

The distance-equalizer sets of graphs were introduced in \cite{Gonzalez-2022}, as a new graph structure that, in a local manner, could be somehow related to privacy in social networks in a similar way as the $k$-antiresolving sets, already known from \cite{Trujillo-Rasua-2016}. A $k$-antiresolving set $S$ of a graph $G$ is understood to have the following property. Given an integer $k\ge 1$, all the vertices not in $S$ can be partitioned into a collection of subsets of cardinality at least $k$ so that vertices belonging to a same set of such collection share the same distance to each vertex of $S$. In contrast with this, a distance-equalizer set is a set $D$ of vertices of $G$, so that for any two vertices $x,y\notin D$ there is a vertex in $w\in D$ such that the distances from $w$ to $x$ and $w$ are the same.

It can be also noted that distance-equalizer sets are a kind of dual notion to the classical resolving sets of graphs, which require the following property. For any two vertices of a graph, there is vertex $w$ in the resolving set such that the distances from $w$ to the two mentioned vertices differ. Resolving sets are nowadays very well studied. To see more information about them, we simply suggest the recent surveys \cite{Tillquist,Kuziak}.

In a formal way, from now on in our whole exposition we consider $G=(V(G), E(G))$ is a connected undirected graph without loops and multiple edges. A set of vertices $D\subseteq V(G)$ is called a \textit{distance-equalizer set} of $G$ if for any two vertices $x,y\notin D$, there is a vertex $w\in D$ such that $d_G(x,w)=d_G(y,w)$, where $d_G(u,v)$ (or simply $d(u,v)$ if $G$ is clear) is the \textit{distance} between $u$ and $v$; and it represents the length of a shortest path between $u$ and $v$, also written as, a shortest $u,v$-path. A distance-equalizer set of the smallest possible cardinality is called a \textit{distance-equalizer basis}. The cardinality of any distance-equalizer basis is called the \textit{equidistant dimension} of $G$, and it is denoted by $\xi(G)$

As already mentioned, this novel parameter  was introduced  in \cite{Gonzalez-2022}, where the authors explored its primary properties and proposed some applications to other problems not necessarily in the context of graph theory. After this seminal paper, some more contributions on the topic have been given in \cite{Gispert,Kratica,Savic}. For instance, it was proved in \cite{Gispert} that finding the equidistant dimension of graphs is NP-hard, which suggest that bounding or computing the value of this parameter for some non trivial families of graphs is worth of attention. In \cite{Gispert}, a detailed study of lexicographic product graphs was also described. In \cite{Kratica,Savic}, families of Johnson and Kneser graphs and of some graphs of convex polytopes were considered. We may remark the interesting fact that, surprisingly, the equidistant dimension of paths $P_n$ was shown to be very challenging to compute in the seminal article  \cite{Gonzalez-2022}, where the authors proved that such parameter for paths $P_n$ relates to 3-AP-free sets and a function, denoted $r(n)$, introduced by Erd\"os and Tur\'an in the article \cite{Erdos}. A subset of integers $S\subseteq \{1,\dots,n\}$ is 3-AP-free if $a+c\ne 2b$ for every distinct terms $a,b,c\in S$. The largest cardinality of a 3-AP-free subset of $\{1,\dots,n\}$ is denoted by $r(n)$.

\subsection{Preliminaries}

We next include some necessary terminology and notations that shall be used in our exposition, as well as, some required known results. First, for any integer $n$, we shall write $[n]=\{1,\dots,n\}$. Also, for the integers $k$ and $q$, by $\mode{k}_q$ we represent the value of $k\pmod q$.

The next result, proved in \cite{Gonzalez-2022}, will be useful along our exposition, while working with bipartite graphs. We indeed show that the bound of such result is achieved in several situations.

\begin{proposition}{\em \cite{Gonzalez-2022}}
\label{PropoBipartiteUPC}
Let $G$ be a bipartite graph with partite sets $A$ and $B$. If $X$ is a distance-equalizer
set of $G$, then $A \subseteq X$ or $B \subseteq X$. Consequently, $\xi(G)\ge \min\{|A|, |B|\}$.
\end{proposition}

Although, it is a very well known fact, we next recall the following observation that shall be (implicitly) used several times along our exposition, and that is it based on the structural properties of bipartite graphs.

\begin{remark} \label{lemmahypercubes}
Let $G$  be  a connected bipartite graph, with partite sets $A$ and $B$.
Two different vertices $v,w\in V(G)$ satisfy that $d(v,w)\equiv0\pmod2$ if and only if $v,w\in A$ or $v,w\in B$.
\end{remark}

The \textit{Cartesian product} of two graph $G$ and $H$, denoted $G\square H$, has vertex set $V(G)\times V(H)$, and two vertices $(g,h),(g',h')$ are adjacent in $G\square H$ if either $g=g'$ and $hh'\in E(H)$; or $h=h'$ and $gg'\in E(G)$. Along our exposition, we  shall use the notation  $(x_v,y_v)$ to refer to a vertex $v=(x,y)\in V(G\square H)$. Also, given a  subset  $S\subseteq V(G\square H)$  we define  the projections sets $X_{S}=\{x_v: \, v\in S\}$ and $Y_{S}=\{y_v: \, v\in S\}$ associated to $S$. In this context, $|X_{S}|=k$ (respectively $|Y_{S}|=k$) means that there are $k$ different vertices in $X_{S}$ (respectively in $Y_{S}$).

In this article we compute the value of the equidistant dimension of several classical Cartesian product graphs. In particular, we prove the following results.

\begin{itemize}
\item For any two dimensional Hamming graph $K_n\square K_m$ with $m,n\ge 2$,
$$
\xi(K_n\square K_m)=\left\{ \begin{array}{ll}
    m; & \mbox{if $n=2$}, \\
    \min\{m,n,5\}; & \mbox{if $\min\{m,n\}\ge 3$.}
\end{array} \right.
$$
\item For any hypercube $Q_n$ of dimension $n\ge 2$ such that  $n\not\equiv0 \text{ mod }4$,
$$\xi(Q_n)=2^{n-1}.$$
Moreover, if $n\equiv0 \text{ mod }4$, then $$2^{n-1}\le\xi(Q_n)\le 2^{n-1}+2^{\frac{n}{2}-2}.$$
\item For any prism graph $C_n\square K_2$ with $n\ge 3$
    $$\xi(C_n\square K_2)=\left\{ \begin{array}{ll}
    \frac{5n-2}{4}; &  \text{ if }  \, \, n\equiv 2 \pmod 4,\\
    n; &  \text{ otherwise.}
    \end{array} \right. $$
\item For any squared grid graph $P_n\square P_n$ with $n\ge 2$,
$$
\xi\left(P_n\square P_n\right)=\left\lceil\frac{n^2}{2}\right\rceil.
$$
\end{itemize}

\section{Two-dimensional Hamming graphs}

The two-dimensional Hamming graph is understood as the Cartesian product of two complete graphs $K_n$ and $K_m$. Studies on its metric dimension and several other distance related parameters of such graph class are very well known in the literature. For instance, a fairly complete study on its metric dimension was considered in \cite{Caceres}. We next address the equidistant dimension of such graphs, and we present them into separated situations in order to simplify the expositions. Along the section, we shall write $V(K_n)=[n]$ and $V(K_m)=[m]$, and also $V(K_n\square K_m)=[n]\times [m]$.

\begin{proposition}
For any integer  $m\ge 2$,
$$\xi(K_2\square K_m)=m.$$
\end{proposition}

\begin{proof}
First, notice that $\xi(K_2\square K_2)=\xi (C_4)=2$. Hence, we assume that $m\ge 3$.
Let $A_1=\{1\}\times [m]$ and let $A_2= \{2\}\times [m]$.  Let $S\subseteq V(K_2\square K_m)$ be a distance-equalizer set of $K_2\square K_m$. Suppose that $|S|<m$.
Thus, it is readily seen that there exists $k\in [m]$ such that $v'=(1,k),v''=(2,k)\in V(K_2\square K_m)\setminus S$. Hence, for each  $s'\in S\cap A_1$ we have $d(v',s')=1\neq 2=d(v'',s')$, as well as, for each $s''\in S\cap A_2$ we have $d(v',s'')=2\neq 1=d(v'',s'')$, which is a contradiction. Therefore, $\xi(K_2\times K_m)= |S|\ge m.$

On the other hand, since $m\ge 3$, for every $x,y\in A_2$ there exists $z\in A_1$ such that $d(x,z)=d(y,z)=2$, which implies that $A_1$ is a distance-equalizer set of $K_2\square K_m$. Thus,  $\xi(K_2\times K_m)\le |A_1|= m$, which completes the equality proof.
\end{proof}

\begin{proposition}
For any integer  $m\ge 3$,
$$
\xi(K_3\square K_m)=3.
$$
\end{proposition}

\begin{proof}
Consider the subset $S=V(K_3)\times \{1\}=\{(1,1),(2,1),(3,1)\}\subseteq V(K_3\square K_m)$.
Notice that for any pair of vertices $v,w\in V(K_3\square K_m)\setminus S$ it holds that $y_v\ne 1$ and $y_w\ne 1$, and  there exists $s\in S$ with $x_s\ne x_v$ and $x_s\ne x_w$. Thus,  $d(v,s)=2=d(w,s)$ and, as a result,  $S$ is a distance-equalizer set of $K_3\square K_m$. Therefore, $\xi(K_3\square K_m)\leq |S|=3$.

Now, we proceed to show that no subset of vertices of cardinality two forms a distance-equalizer set of $K_3\square K_m$. To this end, we consider an arbitrary  subset  $S'=\{a,b\}\subseteq V(K_3\square K_m)$. We differentiate  two cases.

If $y_a\neq y_b$, then we consider a pair of vertices $v,w\in V(K_3\square K_m)\setminus  S'$ such that $x_v=x_w\notin X_{S'}$, $y_v=y_a$ and $y_w=y_b$.
Hence, clearly $d(v,a)=1\neq 2=d(w,a)$ and $d(v,b)=2\neq1=d(w,b)$, which is not possible since $S'$ is a distance-equalizer set.

Now, if $y_a=y_b$, then we consider a pair of vertices $v,w\in V(K_3\square K_m)\setminus  S'$ such that $x_v=x_w\notin  X_{S'}$, $y_v=y_a$ and $y_w\neq y_b$. In such case, it holds $d(v,a)=1\neq 2=d(w,a)$ and $d(v,b)=1\neq 2=d(w,b)$. Thus, we again have a contradiction.

Therefore, $S'$ can not be a distance-equalizer set, which implies that  $\xi(K_3\square K_m)= 3$.
\end{proof}

\begin{proposition}
For any integer $m\ge 4$,
$$
\xi(K_4\square K_m)=4.
$$
\end{proposition}

\begin{proof}
Consider the subset $S=[4]\times \{1\}\subseteq V(K_4\square K_m)$.
Notice that for any pair of vertices $v,w\in V(K_4\square K_m)\setminus S$ we have $y_v\ne 1$ and $y_w\ne 1$, and  there exists $s\in S$ with $x_s\ne x_v$ and $x_s\ne x_w$. Thus,  $d(v,s)=2=d(w,s)$ and, as a result,  $S$ is a distance-equalizer set of  $K_4\square K_m$. Therefore,   $\xi(K_4\square K_m)\leq |S|=4$.

Now, we claim that no subset of vertices of cardinality three forms a distance-equalizer set of $K_4\square K_m$. To show this, we consider an arbitrary  subset  $S'=\{a,b,c\}\subseteq V(K_4\square K_m)$.  Without loss of generality, we differentiate the following cases which are associated to the possible value of $|Y_{S'}|$.

\medskip
\noindent Case 1.
$|Y_{S'}|=1$.
Two vertices $v,w\in V(K_4\square K_m)\setminus  S'$ such that $x_v=x_w\notin X_{S'},$  $y_v=y_a$ and $y_w\neq y_a$ are not equidistant to any vertex of $S'$, because $d(v,s)=1\neq2=d(w,s)$ for every $ s \in S'$.
Hence, in this case, $S'$ is not a distance-equalizer set.

\medskip
\noindent Case 2. $|Y_{S'}|=2$. We assume, without loss of generality, that
$y_a\neq y_b=y_c$. In this case,
two vertices $v,w\in V(K_4\square K_m)\setminus  S'$ such that $x_v=x_w\notin X_{S'}$, $y_v=y_a$ and $y_w=y_b$ are not equidistant to any vertex belonging to $S'$, owing to the fact that $d(v,a)=1\neq2=d(w,a)$ and $d(v,b)=d(v,c)=2\neq1=d(w,b)=d(w,c)$.
Hence,  $S'$ is not a distance-equalizer set either.

\medskip
\noindent Case 3.
$|Y_{S'}|=3$.
If $|X_{S'}|=1$, then we consider a pair of vertices $v,w\in V(K_4\square K_m)\setminus  S'$ such that $x_v\in X_{S'}$,  $x_w \not\in X_{S'} $ and $y_v, y_w\notin Y_{S'}$.
Then, $d(v,a)=d(v,b)=d(v,c)=1\neq2=d(w,a)=d(w,b)=d(w,c)$, so $v$ and $w$ are not equidistant to any vertex belonging to $S$.

Now, if  $|X_{S'}|=2$, we can assume, without loss of generality, that $x_a=x_b\neq x_c$ and consider a pair of vertices $v,w\in V(K_4\square K_m)\setminus  S'$ such that $x_v=x_a$, $y_v=y_c$, $x_w\notin\{x_a,x_c\}$ and $y_w\notin Y_{S'}$.
Hence, $d(v,a)=d(v,b)=d(v,c)=1\neq2=d(w,a)=d(w,b)=d(w,c)$, so $v$ and $w$ are not equidistant to any vertex belonging to  $S'$.
Finally, if $|X_{S'}|=3$, then we consider a pair of vertices $v,w\in V(K_4\square K_m)\setminus  S'$ such that $x_v=x_a$, $y_v=y_b$, $x_w=x_c$ and $y_w\notin Y_{S'}$.
Thus, $d(v,a)=d(v,b)=1\neq2=d(w,a)=d(w,b)$ and $d(v,c)=2\neq1=d(w,c)$.
Hence, again no vertex in $S'$ is equidistant to   $v$ and $w$.   Therefore,   $S'$ is not a distance-equalizer set.

\medskip
\noindent According to the three cases above, we conclude  that no subset of vertices of cardinality three forms a distance-equalizer set of $K_4\square K_m$ as claimed. Therefore,  $\xi(K_4\square K_m)=4$.
\end{proof}


\begin{proposition}
For any pair of integers  $m,n\ge 5$,
$$
\xi(K_n\square K_m)=5.
$$
\end{proposition}

\begin{proof}
Firstly, we prove that $\xi(K_n\square K_m)\leq 5$ by showing that the subset $$S=\{(1,1),(2,1),(3,1),(2,2),(3,3)\}$$   is a distance-equalizer set of $K_n\square K_m$. Let  $v,w\in V(K_n\square K_m)\setminus S$. We consider the cases for  $v=(x_v,y_v),w=(x_w,y_w)\not\in S$.
\vspace{3mm}

\noindent
Case 1.
$y_v=1$ and $y_w=1$.  In this case,  $v$ and $w$ are equidistant to  $s=(1,1)\in S$ , as $d(v,s)=1=d(w,s)$.

\medskip
\noindent Case 2.
$y_v\ne 1$ and $y_w\neq1$. In this case, $v$ and $w$ are equidistant to  $s=(x_s,1)\in S$ with $x_s\in \{1,2,3\}\setminus \{x_v,x_w\}$, as $d(v,s)=2=d(w,s)$.

\medskip
\noindent Case 3.
$y_v=1$ and $y_w\neq1$. If $x_w \in \{2,3 \}$, then $v$ and $w$ are equidistant to  $s=(x_w,1)\in S$ as $d(v,s)=1=d(w,s)$. Now, if $x_w \not \in \{2,3 \}$, then $v$ and $w$ are equidistant to $s=(j,j)\in S$, where $j\in \{2,3\}\setminus \{y_w\}$, as $d(v,s)=2=d(w,s)$.

\medskip
\noindent Thus, we conclude that $S$ is a distance-equalizer set of $K_n\square K_m$ and so $\xi(K_n\square K_m)\leq|S|=5$.

\medskip
Now, we claim that no subset of vertices of cardinality four forms a distance-equalizer set of $K_n\square K_m$. To this end, we consider an arbitrary  subset  $S'=\{a,b,c,d\}\subseteq V(K_n\square K_m)$.  Without loss of generality, we differentiate the following cases which are associated to the possible value of $|Y_{S'}|$.

\medskip
\noindent Case 1'. $|Y_{S'}|=1$.
Let us consider  two different vertices  $v,w\in V(K_n\square K_m)\setminus  S$ such that $x_v,x_w\notin X_{S'}$, $y_v=y_a$ and $y_w\neq y_a$. Since $y_a=y_b=y_c=y_d$, we deduce that $d(v,s)=1\neq2=d(w,s)$ for every $s\in S'$.

\medskip
\noindent Case 2'. $|Y_{S'}|=2$. Without loss of generality, we consider the following subcases.

\medskip
\noindent Subcase 2' (a). $y_a\ne y_b=y_c=y_d$.
The two vertices  $v,w\in V(K_n\square K_m)\setminus  S$ such that $x_v,x_w\notin X_{S'}$, $y_v=y_b$ and $y_w=y_a$ satisfy that $d(v,a)=2\ne 1=d(w,a)$ and that $d(v,b)=d(v,c)=d(v,d)=1\ne 2=d(w,b)=d(w,c)=d(w,d)$. Therefore,  $v$ and $w$ are not equidistant to any vertex belonging to  $S'$.

\medskip
\noindent Subcase 2' (b). $y_a= y_b\ne y_c=y_d$. In this situation we select $x_v,x_w\notin X_{S'}$, $y_v=y_b$ and $y_w=y_c$, to obtain that $d(v,a)=d(v,b)=1 \ne 2=d(w,b)=d(w,a)$ and $d(v,c)=d(v,d)=2\ne 1=d(w,d)=d(w,c)$. Thus, $v$ and $w$ are not equidistant to any vertex belonging to  $S'$.

\medskip
\noindent Case 3'. $|Y_{S'}|=3$.  Without loss of generality, we consider $y_a=y_d$. The subcases $|X_{S'}|=1$ and $X_{S'}|=2$ can be omitted, by symmetry with Cases 1' and 2', respectively. Thus, we restrict ourselves to the following subcases.

\medskip
\noindent Subcase 3' (a).  Without loss of generality, let $x_c=x_d$.
Consider the pair of vertices $v,w\in V(K_n\square K_m)\setminus  S$ such that $x_v=x_b$, $y_v=y_a$, $y_w=y_c$ and $x_w\notin X_{S'}$.
Hence, $d(v,a)=d(v,b)=d(v,d)=1\neq2=d(w,a)=d(w,b)=d(w,d)$ and $d(v,c)=2\neq1=d(w,c)$, and so, $v$ and $w$ are not equidistant to any vertex belonging to  $S'$.

\medskip
\noindent Subcase 3' (b). $|X_{S'}|=4$. We consider the vertices $v,w\in V(K_n\square K_m)\setminus  S$ such that $x_v=x_b$, $y_v=y_a$, $x_w=x_c$ and $y_w\notin Y_{s'}$. Thus,
$d(v,a)=d(v,b)=d(v,d)=1\ne 2= d(w,d)=d(w,b)=d(w,a)$ and $d(v,c)=2\ne 1 =d(w,c)$. Hence, $v$ and $w$ are not equidistant to any vertex belonging to  $S'$.

\medskip
\noindent Case 4'. $|Y_{S'}|=4$. The situations $|X_{S'}|=1$, $|X_{S'}|=2$ and  $|X_{S'}|=3$ can be omitted, by symmetry with Cases 1', Case 2' and Subcase 3' (b), respectively. Hence, if  $|X_{S'}|=4$, then we select two vertices $v,w\in V(K_n\square K_m)\setminus  S$ such that $x_v=x_a$, $y_v=y_b$, $x_w=x_c$ and $y_w=y_d$.
Thus $d(v,a)=d(v,b)=1\neq 2=d(w,a)=d(w,b)$ and $d(v,c)=d(v,d)=2\neq 1=d(w,c)=d(w,d)$. Therefore,  $v$ and $w$ are not equidistant to any vertex belonging to  $S'$.

According to all the conclusions obtained in the cases above, we deduce that no subset of vertices of cardinality four forms a distance-equalizer set of $K_n\square K_m$ as claimed. Therefore,  $\xi(K_n\square K_m)=5$.
\end{proof}

Once considered all the situations for the graphs $K_n\Box K_m$ for $n,m\ge 2$, from all the results of this section, we have the following theorem.

\begin{theorem}
For any pair of integers  $m,n\ge 2$,
$$
\xi(K_n\square K_m)=\left\{ \begin{array}{ll}
    m; & \mbox{if $n=2$}, \\
    \min\{m,n,5\}; & \mbox{if $\min\{m,n\}\ge 3$.}
\end{array} \right.
$$
\end{theorem}

\section{Hypercubes}\label{sec-hyper}

We shall use the (commonly known) binary notation of the vertices of $Q_n$, which consist on denoting each vertex $v\in V(Q_n)$ with a string of $n$ binary bits, i.e., $v=v_1v_2\dots v_i\dots v_n$, where $v_i\in\{0,1\}$ for every $i\in[n]$.
If two vertices have only one bit different, then they are adjacent in $Q_n$.
The hypercubes are bipartite graphs with bipartition sets $A$ and $B$, where $|A|=|B|$. Using this binary notation, it can be seen that
$$A=\left\{v\in V(Q_n):\sum_{i=1}^{n}v_i\equiv0\pmod{2
}\right\} \text{ and } B=\left\{v\in V(Q_n):\sum_{i=1}^{n}v_i\equiv1\pmod{2}\right\}$$
are the bipartition sets of $V(Q_n)$. For our purposes, the vertices of the subset $A$ will be called $even$, while the vertices of the subset $B$ will be called $odd$. It is commonly known that the distance between any two vertices of $Q_n$ is the number of bits where they differ. Also, notice that $Q_n$ is a $2$-antipodal graph, and we shall denote the antipodal vertex of any vertex $v\in V(Q_n)$ as $\Bar{v}$. Obviously, $\bar{\bar{v}}=v$ and $d(v,\Bar{v})=n$ for every $v\in V(Q_n)$. In addition, we may recall that $Q_n$ is indeed the Cartesian product of $n$ graphs isomorphic to $K_2$.

\begin{theorem}
\label{th-hyper}
Let $Q_n$ be the hypercube of dimension $n\ge 2$.
\begin{itemize}
\item[{\em (i)}] If $n\not\equiv0 \text{ mod }4$, then $\xi(Q_n)=2^{n-1}$.
\item[{\em (ii)}] If $n\equiv0 \text{ mod }4$, then $2^{n-1}\le\xi(Q_n)\le 2^{n-1}+2^{\frac{n}{2}-2}$
\end{itemize}
\end{theorem}

\begin{proof}
Firstly, by Proposition~\ref{PropoBipartiteUPC}, we deduce the lower bound $\xi(Q_n)\ge \min\{|A|, |B|\}=|A|=|B|=2^{n-1}$, which is valid for any integer $n\ge 2$. We next differentiate two cases depending on the congruency of $n$ modulo $4$.

\vspace{3mm}
\noindent Case 1. $n\not\equiv0\pmod4$. To obtain the required equality, we need to prove the upper bound $\xi(Q_n)\le 2^{n-1}$. To this end, we next show that $B$ is a distance-equalizer set of $Q_n$.
Let  $v,w\in A$ be two different vertices. First, we consider the case $d(v,w)=d<n$. Notice that $d$ is even, by Remark~\ref{lemmahypercubes}. Let $\{i_1,\dots,i_d\}$ be the set of indices such that $v_i\neq w_i$ for every $ i\in\{i_1,\dots,i_d\}$. Hence, we consider the vertex $z=z_1\dots z_i\dots z_n$, defined as follows: $z_i=v_i=w_i$ for every $ i\notin\{i_1,\dots,i_d\}$,
$z_i=v_i$ for every $ i\in\{i_1,\dots,i_{\frac{d}{2}}\}$ and $z_i=w_i$ for every $ i\in\{i_{\frac{d}{2}+1},\dots,i_d\}$. Notice that such a vertex exists due to the parity of $d$. Also, it is readily seen that $v$ and $w$ are equidistant to $z$, as $d(v,z)=\frac{d}{2}=d(w,z)$. 
Thus, if  $d\equiv 2\pmod 4$, then  $\frac{d}{2}$ is odd, and so Remark~\ref{lemmahypercubes} leads to $z\in B$, as required.

Now, if $d\equiv 0\pmod 4$, then it holds that $z\notin B$ (that is $z\in A$).
However,  since   $d(v,w)=d<n$, there exists at least one index $j\in\{1,\dots,n\}\setminus \{i_1,\dots,i_d\}$. We then consider the vertex  $z^*$ for which $z^*_i=z_i$ for every  $ i\in\{1,\dots,n\}\setminus \{ j\}$ and $z^*_j=(z_j+1)\pmod 2$. Hence, it happens that $d(v,z^*)=d(v,z)+1=d(w,z)+1=d(w,z^*)$, and that $z^*\in B$, which is equidistant to $v,w$, as required.

\smallskip
On the other hand, assume $v,w\in A$ satisfy that $d(v,w)=n$.
In this case,  $v$ and $w$ are equidistant to  $u=v_1\dots v_{\frac{n}{2}}w_{\frac{n}{2}+1}\dots w_n$, as $d(v,u)=\frac{n}{2}=d(w,u)$.
Moreover, since $n$ is even and  $n\not \equiv0\pmod 4$, it must occur that  $n\equiv2\pmod4$. Therefore, $d(v,u)=\frac{n}{2}$ is odd; and so, by Remark \ref{lemmahypercubes}, we conclude that $u\in B$, as required.

In summary, we have have proved that, if $n\not\equiv0\pmod4$, then $B$ is a distance-equalizer set of $Q_n$ and, as a consequence, $\xi(Q_n)\leq |B|\le 2^{n-1}$, which completes the proof of the equality for $n\not\equiv0\pmod4$.

\medskip
\noindent Case 2.
$n\equiv0\pmod4$. We first recall that $Q_n=Q_{\frac{n}{2}}\Box Q_{\frac{n}{2}}$. Let $G_1$ be the subgraph of $Q_n$ induced by the set $V(G_1)=\{0,1\}^{\frac{n}{2}}\times \{0\}^{\frac{n}{2}}$ which is indeed isomorphic to  $Q_{\frac{n}{2}}$.
Let $A^*\subseteq A\cap V(G_1)$ be the subset of  even vertices of $Q_n$ which are in the subgraph $G_1$, and whose first bit is equal to zero. That is,  a vertex $z=z_1z_2\dots z_n\in A^*$ whenever $z\in A$ and $z_i=0$ for every $i\in \{1\}\cup \{\frac{n}{2}+1, \dots, n\}$. Notice that $|A^*|=2^{\frac{n}{2}-2}$. We claim that $B\cup A^*$ is a distance-equalizer set of $Q_n$. To this end, we first observe that if $v,w\notin B\cup A^*$ are two vertices with $d(v,w)<n$, then, by using some similar arguments as the ones used in Case 1, there exists a vertex $z\in B$ such that $d(v,z)=d(w,z)$, and we are done with this situation. Hence, it only remains to consider the case $d(v,w)= n$ (i.e., vertices that are antipodal in $Q_n$ that are not in $B\cup A^*$).

We proceed to show that for every $v\in A$ such that $\{v,\bar{v}\} \cap  A^*=\varnothing $,
there exists $z\in A^*$ such that $d(v,z)=\frac{n}{2}=d(\bar{v},z)$.
To this end, we define  $k_v$ as the maximum number of bits equal to one, among the second half of bits of $v$ and $\bar{v}$, i.e.,
$$k_v=\max \left\{\sum_{i=\frac{n}{2}+1}^n v_i,   \sum_{i=\frac{n}{2}+1}^n\left((v_i+1)\hspace{-0.4cm}\pmod2\right)\right\}.$$
 We can assume, without loss of generality, that  the maximum is reached at $v$. That is, $k_v=\sum_{i=\frac{n}{2}+1}^n v_i$. We differentiate the following three subcases.

\medskip
\noindent Subcase 2.1. $k_v=\frac{n}{2}$. Notice that  $v_1=0$, as $\bar{v} \not\in  A^*$. We  take $z\in A^* $ defined as $z_i=v_i$  for every   $i\in [\frac{n}{2}]$ while $z_i=0$ for every $i\in  \{\frac{n}{2}+1, \dots, n\}$. Thus,  $d(v,z)=\frac{n}{2}=d(\bar{v},z)$.

\medskip
\noindent Subcase 2.2. $k_v<\frac{n}{2}$ and $v_1=0$.
We define  $z_i=v_i$ for every $i\in [k_v]$, $z_i=\mode{v_i+1}_2$ for every $i\in \{k_v+1,\dots, \frac{n}{2}\}$, and  $z_i=0$ for every $i\in \{\frac{n}{2}+1, \dots, n\}$. Hence, again $d(v,z)=\frac{n}{2}=d(\bar{v},z)$ and $z\in A^* $.

\medskip
\noindent Subcase 2.3. $k_v<\frac{n}{2}$ and $v_1=1$. We take  $z_i=\mode{v_i+1}_2$ for every $i\in [\frac{n}{2}-k_v]$, $z_i=v_i$ for every $i\in \{\frac{n}{2}-k_v+1,\dots ,  \frac{n}{2}\}$, and  $z_i=0$ for every $i\in \{\frac{n}{2}+1, \dots, n\}$. As above,  $z\in A^* $ and $d(v,z)=\frac{n}{2}=d(\bar{v},z)$.

Therefore, $B\cup A^*$ is a distance-equalizer set of $Q_n$, which implies that $$\xi(Q_n)=|S|=|B|+|A\cap S|\le |B| +|A^*|=2^{n-1}+ 2^{\frac{n}{2}-2},$$
and so, the upper bound is proved.
\end{proof}

\section{Prism of a cycle}

Given a graph $G$, the \textit{prism graph} of $G$ is known to be the Cartesian product of $G$ and $K_2$. In this section we consider the case when $G$ is a cycle $C_n$. We shall present the proof of our result in a sequence of several propositions that consider different situations with respect to $n$.

\begin{proposition}
For any odd integer $n\ge 3$, $$\xi(C_n\square K_2)=n.$$
\end{proposition}

\begin{proof}
Let $S_i=[n]\times \{i\}$, with $i\in [2]$. We show that $S_2$ is a distance-equalizer set of $C_n\square K_2$. To this end, let $x,w\in S_1$ be two arbitrary vertices. If $x_v+x_w\equiv 0 \pmod  2$, then for $s=(\frac{x_v+x_w}{2},2)\in S_2$ it holds that $d(x,s)=\frac{x_v+x_w}{2}+1=d(w,s)$. Now, since $n$ is odd, if  $x_v+x_w\equiv 1 \pmod  2$, then for $x_s=\frac{ x_v+x_w+n}{2}  \pmod  n$ it follows that $s=(x_s,2)\in S_2$ and that $d(x,s)=x_s+1=d(w,s)$. Hence, $S_2$ is a distance-equalizer set of $C_n\square K_2$, and so, $\xi(C_n\square K_2)\le |S_2|=n.$

It remains to show that $\xi(C_n\square K_2)\ge n.$
Suppose, to the contrary, that there exists a distance-equalizer set $S\subseteq V(C_n\square K_2)$ such that $|S|\le n-1$. In such a case, there exist two different vertices $v,w\in V(C_n\square K_2)\setminus S$ such that $x_v=x_w$. Notice that $y_v\ne y_w$. Thus, there exists $s=(x_s,y_s)\in S$ such that $d(v,s)=d(w,s)$.
Hence,
\begin{align*}
d_{K_2}(y_v,y_s)&=d(v,s)-d_{C_n}(x_v,x_s)\\
							&=d(w,s)-d_{C_n}(x_w,x_s)\\
							&=d_{K_2}(y_w,y_s),
\end{align*}
which is a contradiction, as $y_v,y_w,y_s\in \{1,2\}$ and $y_v\ne y_w$. Therefore,
$\xi(C_n\square K_2)\ge n$ and, as a consequence, the result follows.
\end{proof}

In order to prove the remaining cases for the prism $C_n\Box K_2$, we need the following technical lemmas.

\begin{lemma}\label{L2}
Let $n$ be an even integer. Let $B=\{v\in V(C_n\square K_2):\, \, x_v+y_v\equiv1\pmod{2}\}$ and $A=V(C_n\square K_2)\backslash B$.
If $ v,w\in A$ with $d_{C_n}(x_v,x_w)\equiv 0\pmod{2}$, then there exists $ s\in B$ such that $d(v,s)=d(w,s)$.
\end{lemma}
\begin{proof}
Given a pair of vertices $v,w\in A$ such that $d_{C_n}(x_v,x_w)\equiv 0\pmod{2}$,
we have $y_v=y_w$ and $s=\left(\frac{x_v+x_w}{2},\mode{\frac{x_v+x_w}{2}}_2+1\right)\in B$.
Hence,
\begin{align*}
d(v,s)&=d_{C_n}(x_v,x_s)+d_{K_2}(y_v,y_s)\\
		&=d_{C_n}\left(x_v,\frac{x_v+x_w}{2}\right)+d_{K_2}(y_v,y_s)\\
  	&=d_{C_n}\left(x_w,\frac{x_v+x_w}{2} \right)+d_{K_2}(y_w,y_s)\\
&=d_{C_n}(x_w,x_s)+d_{K_2}(y_w,y_s)\\
&=d(w,s).
\end{align*}
Therefore, the result follows.
\end{proof}

\begin{lemma}\label{L3}
Let $n$ be an even integer. Let $B=\{v\in V(C_n\square K_2):\, \, x_v+y_v\equiv1\pmod{2}\}$ and $A=V(C_n\square K_2)\backslash B$.
If  $ v,w\in A$ with $d_{C_n}(x_v,x_w)\equiv 1\pmod{4}$, then there exists $ s\in B$ such that $d(v,s)=d(w,s)$.
\end{lemma}

\begin{proof}
Given a pair of vertices $v,w\in A$ such that $d_{C_n}(x_v,x_w)\equiv 1\pmod{4}$,
we have that $y_v\ne y_w$ and $x_v\not\equiv y_w\pmod{2}$.  Without loss of generality, we can assume that $x_v<x_w$. Notice also that $\frac{x_w-x_v-1}{2}\equiv 0 \pmod 2$. Thus, $s=\left(x_v+\frac{x_w-x_v-1}{2},y_w  \right)\in B$ and
\begin{align*}
d(s,v)&=d_{C_n}(x_s,x_v)+d_{K_2}(y_s,y_v)\\
			&=\frac{x_w-x_v-1}{2}+1\\
			&=x_w-x_s \\
			&=d_{C_n}(x_s,x_w)+d_{K_2}(y_s,y_w)\\
			&=d(s,w).
\end{align*}
Therefore, the result follows.
\end{proof}

\begin{proposition}\label{EQDIM_CkP2_k_0_mod4}
Let $n\ge 4$ be an integer. If $ n\equiv0\pmod{4}$, then
$$
\xi(C_n\square K_2)=n.
$$
\end{proposition}
\begin{proof}
Since $C_n\square K_2$ is a regular and bipartite graph, the partite sets have cardinality $n$. Hence, by Proposition~\ref{PropoBipartiteUPC} we deduce that $\xi(C_n\square K_2)\ge n .$

It remains to show that $\xi(C_n\square K_2)\le n$ is satisfied.
To this end, we claim that the set of vertices $S=\left\{s \in V(C_n\square K_2):\, \, x_s+y_s \equiv 1 \pmod{2}\right\}$ is a distance-equalizer set.
By the symmetry of $C_n\square K_2$, we fix an arbitrary vertex $v\in V(C_n\square K_2)\backslash S$ and then prove that for any $  w\in V(C_n\square K_2)\backslash\{S\cup\{v\}\}$ there exists a vertex $s\in S$ such that $d(v,s)=d(w,s)$.

Without loss of generality, we consider $v=(1,1)$ and differentiate the following cases according to the possibilities for $x_w$.

\medskip
\noindent Case 1.
$x_w\equiv 0\pmod{4}$.
Consider the vertex $s=\left(\frac{n+x_w}{2},1\right)$.
Notice that $ s\in S$, as $x_s=\frac{n}{2}+\frac{x_w}{2}\equiv0\pmod{2}$ and $y_s=1$.
Moreover,
\begin{align*}
d(v,s)&=  d_{C_n}(x_v,x_s)+d_{K_2}(y_v,y_s)\\
&=\min\{x_s-1,n-(x_s-1)\}\\
&=\frac{n-x_w+2}{2}\\
		&=\min\left\{\frac{n-x_w}{2}, \frac{n+x_w}{2}  \right\}+1\\
		&=\min\left\{x_s-x_w, n-(x_s-x_w)  \right\}+1\\
		&= d_{C_n}(x_w,x_s)+d_{K_2}(y_w,y_s)\\
		&=d(w,s).
\end{align*}

\noindent Case 2.
$x_w\equiv 2\pmod{4}$.
In this situation,  $d_{C_n}(x_v,x_w)\equiv1\pmod{4}$ and by Lemma~\ref{L3} we have that there exists $ s\in S$ such that $d(v,s)=d(w,s)$.

\medskip
\noindent Case 3.
$x_w\equiv 1\pmod{2}$.
Hence,  $d_{C_n}(x_v,x_w)\equiv0\pmod{2}$ and by Lemma~\ref{L2}  there exists  $ s\in S$ such that $d(v,s)=d(w,s)$.

\medskip
\noindent According to the conclusions from the cases above, we deduce that $S$ is a distance-equalizer set of $ C_n\square K_2$, and so, $\xi(C_n\square K_2)\leq n$, which leads to the desired equality.
\end{proof}

From now on, we consider the case of $C_n\square K_2$ with $n\equiv2\pmod{4}$, where we also need some extra terminology and technical lemmas. Given two vertices $v,w$ of a graph $G$, the \textit{bisector} of $v$ and $w$ is denoted by $B_{v|w}$, i.e.,
$$B_{v|w}=\{x\in V(G): \, \, d(v,x)=d(w,x)\}.$$

\begin{lemma}\label{L1}
Let $n\ge 6$ be an  integer such that $n\equiv 2 \pmod 4$. Let $B=\{v\in V(C_n\square K_2):\, \, x_v+y_v\equiv1\pmod{2}\}$ and $A=V(C_n\square K_2)\backslash B$.
If $v,w\in A$  with $d_{C_n}(x_v,x_w)\equiv 3\pmod{4}$, then $B_{v|w}\subseteq A$.
\end{lemma}

\begin{proof}
By the symmetry of $C_n\square K_2$, without loss of generality, we fix $v=(1,1)$ and $w=(x_w,2)$ such that $d_{C_n}(1,x_w)\equiv3\pmod{4}$. Notice that $ x_w\equiv0\pmod{4}$.
With these assumptions in mind, we only need to observe that
 $$B_{v|w}=\left\{\left(\frac{x_w}{2},2\right), \left(\frac{x_w}{2}+1,1\right),\left(\mode{\frac{x_w}{2}+\frac{n}{2}}_n,1\right),\left(\mode{\frac{x_w}{2}+\frac{n}{2}+1}_n,2\right)\right\}\subseteq  A.$$
\end{proof}

As above, we assume that $n\equiv2\pmod{4}$, and define the subsets of vertices $B=\{v\in V(C_n\square P_2):x_v+y_v\equiv1\pmod{2}\}$ and $A=V(C_n\square P_2)\setminus  B$.
Furthermore, we define
$$Q=\left\{\{v,w\}\subseteq A:\,\,  d_{C_n}(x_v,x_w)=3\right\}.$$ Clearly, if $\{v,w\}\in Q$, then, without loss of generality, we take $w=(\mode{x_v+3}_n,\mode{y_v}_2+1)$.  Hence, if $s\in B_{v|w}$, then either $d(v,s)=d(w,s)=2$ or $d(v,s)=d(w,s)=\frac{n}{2}-1$ and
$$B_{v|w}=\left\{\left(\mode{x_v+1}_n,\mode{y_v}_2+1\right),(\mode{x_v+2}_n,y_v),\left(\mode{x_v+1+\frac{n}{2}}_n,y_v\right),\left(\mode{x_v+2+\frac{n}{2}}_n,\mode{y_v}_2+1\right)\right\}.$$
With this notation in mind, we state the following necessary lemmas.

\begin{lemma}\label{starlemma}
Let $n\ge 6$ be an integer such that $n\equiv 2 \pmod 4$.  Let $\{v,w\}\in Q$ be such that $x_w=\mode{x_v+3}_n$ and let $ l\in \{0,\dots, \lfloor\frac{n}{4}\rfloor-1\}$. If $\Tilde{v}_l=(\mode{x_v-2l}_n,y_v)$ and $\Tilde{w}_l=(\mode{x_w+2l}_n,y_w)$, then $B_{v|w}= B_{\Tilde{v}_l|\Tilde{w}_l}$.
\end{lemma}

\begin{proof}
Without loss of generality, by the symmetry of $C_n\square K_2$, we restrict ourselves to $v=(n-1,1)$ and $w=(2,2)$.
Notice that
$$B_{v|w}=\left\{(1,1),(n,2),\left(\frac{n+2}{2} ,2\right),\left( \frac{n}{2},1 \right)\right\}\subseteq A.$$
We take $s=(1,1)$.
Hence,  for every  $ l\in \{0,\dots, \lfloor\frac{n}{4}\rfloor-1\}$,
\begin{align*}
d(\Tilde{v}_l,s)&=\min\{|\mode{x_v-2l}_n-x_s|,n-|\mode{x_v-2l}_n-x_s|\}+\mode{y_v+y_s}_2\\
&=\min\{\mode{n-1-2l}_n-1,n+1-\mode{n-1-2l}_n\}+\mode{1+1}_2\\
&=\min\{n-2l-2,2+2l\}\\
&=2+2l.
\end{align*}
On the other hand,
\begin{align*}
d(\Tilde{w}_l,s)&=\min\{|\mode{x_w+2l}_n-x_s|,n-|\mode{x_w+2l}_n-x_s|\}+\mode{y_w+y_s}_2\\
&=\min\{\mode{2+2l}_n-1,n+1-\mode{2+2l}_n\}+\mode{2+1}_2\\
&=\min\{1+2l,n-2l-1\}+1\\\
&=2+2l.
\end{align*}
Thus,  $d(\Tilde{v}_l,s)=2+2l=d(\Tilde{w}_l,s)$  for every $l\in\{0,\dots,\lfloor\frac{n}{4}\rfloor-1\}$.   The remaining three cases where  $s\in B_{v|w}\setminus \{(1,1)\}$ are analogous to the previous one. Hence, $B_{v|w}\subseteq B_{\Tilde{v}_l|\Tilde{w}_l}$ and, since  $| B_{\Tilde{v}_l|\Tilde{w}_l}|=4$, we conclude that $B_{v|w}= B_{\Tilde{v}_l|\Tilde{w}_l}$.
\end{proof}

\begin{lemma}\label{lemma4vertex}
For each $s\in A$, there exist exactly four pairs of vertices $\{v,w\}\in Q$,  such that $s\in B_{v|w}$.
\end{lemma}

\begin{proof}
Let $\{v,w\}\in Q$  where  $w=(\mode{x_v+3}_n,\mode{y_v}_2+1)$.
If $s\in B_{v|w}$, then either $d(v,s)=d(w,s)=2$ or $d(v,s)=d(w,s)=\frac{n}{2}-1$. Without loss of generality, we fix $s=(1,1)$. In order to determine the coordinates of $v$ and $w$, we differentiate the following cases for $a\in\{ v,w \}$.

\medskip
\noindent Case 1. $d(s,a)=2$. In this case, $2=d_{C_n}(x_a,1)+\mode{y_a+1}_2$. Thus, if $y_a=1$, then
$2=d_{C_n}(x_a,1)=\min\{x_a-1, n-x_a+1\}$, and so $x_a=3$ or $x_a=n-1$. Now, if  $y_a=2$, then   $1=d_{C_n}(x_a,1)=\min\{x_a-1, n-x_a+1\}$, and so $x_a=2$ or $x_a=n$. Therefore, since $w=(\mode{x_v+3}_n,\mode{y_v}_2+1)$, there are only two possibilities, i.e.,
$v=(n-1,1)$ and $w=(2,2) $ or $v=(n,2)$ and $w=(3,1) $.

\medskip
\noindent Case 2. $d(s,a)=\frac{n}{2}-1$. In this case, $\frac{n}{2}-1=d_{C_n}(x_a,1)+\mode{y_a+1}_2$. Thus, if $y_a=1$, then
$\frac{n}{2}-1=d_{C_n}(x_a,1)=\min\{x_a-1, n-x_a+1\}$, and so $x_a=\frac{n}{2}$ or $x_a=\frac{n}{2}+2$. Now, if  $y_a=2$, then   $\frac{n}{2}-2=d_{C_n}(x_a,1)=\min\{x_a-1, n-x_a+1\}$, and so $x_a=\frac{n}{2}-1$ or $x_a=\frac{n}{2}+3$. Therefore, since $w=(\mode{x_v+3}_n,\mode{y_v}_2+1)$, there are only two possibilities, i.e.,
$v=(\frac{n}{2},1)$ and $w=(\frac{n}{2}+3,2) $ or $v=(\frac{n}{2}-1,2)$ and $w=(\frac{n}{2}+2,1) $.

According to the two cases above, the result follows.
\end{proof}

\begin{proposition}\label{lowerboundpropCkP2}
For any integer $n\ge 6 $ with $n\equiv 2 \pmod 4$,
$$
\xi(C_n\square K_2)\ge\frac{5n-2}{4}.
$$
\end{proposition}

\begin{proof}
Let $S$ be a $\xi(C_n\square K_2)$-set.
By using Proposition~\ref{PropoBipartiteUPC}, we consider $S$ of the form $S=B\cup S^*$, where $S^*\subseteq A$.
By Lemma \ref{L1}, $B_{v|w}\subseteq A$ for every pair $\{v,w\}\in Q$.
Hence, all pairs $\{v,w\}\in Q$ must be equalized by vertices in $S^*$ or $\{v,w\}\cap S^*\neq\emptyset$.
Suppose that $|S^*|\le \lfloor\frac{n}{4}\rfloor-1=\frac{n-2}{4}-1=\frac{n-6}{4}$.
Thus, if $n=6$, then $S=B$ and so $v=(1,1) $ and $w=(4,2) $ are not equalized by any vertex in $S$, which is a contradiction. From now on, we assume that $n\ge 10$, and so $|Q|=n$.
Thus, Lemma \ref{lemma4vertex} implies that there are at most $4\cdot|S^*|\le n-6<n=|Q|$ pairs of vertices of $Q$ which are equalized by the vertices of $S^*$.
Lemma~\ref{L1} also implies that pairs of vertices in $Q$ can only be equalized by vertices of $A\cap S=S^*$, so there will be at least $6$ pairs of vertices in $Q$ which are not equalized by any vertex in $S$.
Notwithstanding that, this fact does not imply that $S$ is not a distance-equalizer set, as those pairs of vertices that are not equalized could have at least one element in $S^*$.
Let $\{a,b\}\in Q$ be a  set of two vertices which are not equalized by any vertex in $S^*$, where  $\mode{x_a+3}_n=x_b$.
Lemma~\ref{starlemma} implies that the pairs $\Tilde{a}_l=(\mode{x_a-2l}_n,y_a)$ and $\Tilde{b}_l=(\mode{x_b+2l}_n,y_b)$ are not equalized by the vertices in $S^*$ for every $l\in\{0,\dots,\lfloor\frac{n}{4}\rfloor-1\}$.
Consequently, $\{\Tilde{a}_l,\Tilde{b}_l\}\cap S^*\neq\emptyset$ for every $l\in\{0,\dots,\lfloor\frac{n}{4}\rfloor-1\}$.
However, it is readily seen that $\Tilde{a}_i\neq\Tilde{a}_j$ and $\Tilde{b}_i\neq\Tilde{b}_j$, for any $i\neq j$.
Moreover, it also stands that $\Tilde{a}_i\neq\Tilde{b}_j$ for every $i,j\in\{0,\dots,\lfloor\frac{n}{4}\rfloor-1\}$, due to having $y_a\neq y_b$ as $\{a,b\}\in Q$.
As a result, to satisfy $\{\Tilde{a}_l,\Tilde{b}_l\}\cap S^*\neq\emptyset$ for every  $l\in\{0,\dots,\lfloor\frac{n}{4}\rfloor-1\}$ it is necessary to have at least $\left|\{0,\dots,\lfloor\frac{n}{4}\rfloor-1\}\right|=\lfloor\frac{n}{4}\rfloor$ different vertices in $S^*$, whereas $|S^*|\le \lfloor\frac{n}{4}\rfloor-1$, which is a contradiction.
Therefore, $|S^*|\ge \lfloor\frac{n}{4}\rfloor$ which implies that  $\xi(C_n\square K_2)\ge|S|\ge n+\lfloor\frac{n}{4}\rfloor=\frac{5n-2}{4}$.
\end{proof}

\begin{lemma}\label{lemmaUpperboundResto2}
Let   $n\ge 6 $ be an integer with $n\equiv 2 \pmod 4$.  Let
, and let  $a=(\frac{n}{2}-2,1)$,  $b=(\frac{n}{2}+1,2)$, $t=(n-2,2)$, $u=(1,1)$ and $S^*=\left\{(2i-1,1): \, \, i\in \left\{1,\dots,\frac{n-2}{4}\right\}\right\}$.
 Then for every $ \{v,w\}\in Q\setminus \{\{a,b\},\{t,u\}\}$ there exists $ s\in S^*$ such that $d(v,s)=d(w,s)$.
\end{lemma}

\begin{proof}
 We differentiate three cases for pairs $\{v,w\}\in Q\setminus \{\{a,b\},\{t,u\} \}$.

 \vspace{3mm}

\noindent Case 1. $x_v\in \{1, \dots, \frac{n}{2}-3\}$ and $x_w=x_v+3$.
The pairs $\{v,w\}$ associated to this case can be described as $\{v_i,w_i\}$ where $v_i=(i,\mode{i+1}_2+1)$,  $w_i=(i+3,\mode{i}_2+1)$ and $ i\in  \{1, \dots, \frac{n}{2}-3\}$.
If $\mode{i}_2=0$,then $s=(x_{v_i}+1,1)\in S^*$ and $d(v_i,s)=d_{C_n}(x_{v_i},x_s)+d_{K_2}(y_{v_i},y_s)=1+1=2+0=d_{C_n}(x_{w_i},x_s)+d_{K_2}(y_{w_i},y_s)=d(w_i,s)$.
Now, if  $\mode{i}_2=1$, then $s=(x_{v_i}+2,1)\in S^*$ and $d(v_i,s)=d_{C_n}(x_{v_i},x_s)+d_{K_2}(y_{v_i},y_s)=2+0=1+1=d_{C_n}(x_{w_i},x_s)+d_{K_2}(y_{w_i},y_s)=d(w_i,s)$.

\vspace{3mm}

\noindent Case 2. $x_v\in\{n-1,n\}$ and $x_w=x_v+3$.
If ($v=(n-1,1)$ and $w=(2,2)$) or ($v=(n,2)$ and $w=(3,1)$) then for  $s=(1,1)\in S^*$ we have $d(v,s)=2=d(w,s)$.

\vspace{3mm}

\noindent Case 3. $x_v\in\{\frac{n}{2}-1,\dots,n-3\}$ and $x_w=x_v+3$.
In this case, the pairs $\{v,w\}$  can be characterized by   pairs $\{v_j,w_j\}$ and $\{v'_j,w'_j\}$, where  $v_j=(\frac{n}{2}+2j-3,2)$, $w_j=(\frac{n}{2}+2j,1)$,  $v'_j=(\frac{n}{2}+2j-2,2)$ and  $w'_j=(\frac{n}{2}+2j+1,1)$ with $j\in  \left\{1,\dots,\frac{n-2}{4}\right\}$.
Hence, $s_j=(2j-1,1)\in S^*$  for every $j \in \left\{1,\dots,\frac{n-2}{4}\right\}$   and
 $d(v_j,s_j)=d_{C_n}(\frac{n}{2}+2j-3,2j-1)+d_{K_2}(y_{v_j},y_{s_j})=(\frac{n}{2}-2)+1=(\frac{n}{2}-1)+0=d_{C_n}(\frac{n}{2}+2j,2j-1)+d_{K_2}(y_{v_j},y_{s_j})=d(w_j,s_j)$ and also
  $d(v'_j,s_j)=d_{C_n}(\frac{n}{2}+2j-2,2j-1)+d_{K_2}(y_{v'_j},y_{s_j})=(\frac{n}{2}-1)+0=(\frac{n}{2}-2)+1=d_{C_n}(\frac{n}{2}+2j+1,2j-1)+d_{K_2}(y_{w'_j},y_{s_j})=d(w'_j,s_j).$

\noindent According to the three cases above, the result follows.
\end{proof}

\begin{proposition}\label{upperboundpropCkP2}
For any integer $n\ge 6 $ with $n\equiv 2 \pmod 4$,
$$
\xi(C_n\square K_2)\le \frac{5n-2}{4}.
$$
\end{proposition}

\begin{proof}
 As above, let   $B=\{v\in V(C_n\square K_2):\, \, x_v+y_v\equiv1\pmod{2}\}$ and $A=V(C_n\square K_2)\setminus B$.
We proceed to show that $S=B\cup S^*$ is a distance-equalizer set, where $$S^*=\left\{(2i-1,1): \, \, i\in \left\{1,\dots,\frac{n-2}{4}\right\}\right\}.$$

If $ v,w\in A$ with $d_{C_n}(x_v,x_w)\not\equiv 3\pmod{4}$, then by Lemmas~\ref{L2} and  \ref{L3} there exists  $ s\in B\subseteq S$ such that $d(v,s)=d(w,s)$.
From now on we assume that $v,w\in A $ with $d_{C_n}(x_v,x_w)\equiv 3\pmod{4}$. With these assumptions in mind,  we proceed to show that either there exists $s\in S^*$ such that $d(v,s)=d(w,s)$ or $\{v,w\}\cap S^*\ne \varnothing$.
Notice that  $d_{C_n}(x_v,x_w)\equiv 3\pmod{4}$, implies that we  can  identify $\{v,w\}$  with a pair of vertices $\{\Tilde{v}_l,\Tilde{w}_l\}$ such that $\{\Tilde{v}_0,\Tilde{w}_0\}\in Q$,  where $\Tilde{v}_l=(\mode{x_v-2l}_n,y_v)$,  $\Tilde{w}_l=(\mode{x_w+2l}_n,y_w)$ and  $l\in\{0,\dots,\frac{n-6}{4}\}$.

 First, we are going to analyze the cases of pairs  $ \{\Tilde{v}_0,\Tilde{w}_0\}\in Q\setminus \{\{a,b\},\{t,u\}\}$, where
 $a=(\frac{n}{2}-2,1)$,  $b=(\frac{n}{2}+1,2)$, $t=(n-2,2)$ and $u=(1,1)$.
 By Lemma~\ref{lemmaUpperboundResto2}, if $ \{\Tilde{v}_0,\Tilde{w}_0\}\in Q\setminus \{\{a,b\},\{t,u\}\}$, then  there exists $ s\in S^*$ such that $d(\Tilde{v}_0,s)=d(\Tilde{w}_0,s)$. Thus, by Lemma~\ref{starlemma}, $d(v,s)=d(w,s)$.

 On the other hand, assume that $\{\Tilde{v}_0,\Tilde{w}_0\}=\{a,b\}$ or $\{\Tilde{v}_0,\Tilde{w}_0\}=\{t,u\}$. In both cases we have that $\{v,w\}\cap S^*\ne \varnothing$. That is, $a=(\frac{n}{2}-2,1)\in S^*$ and also $\Tilde{a}_l=(\frac{n}{2}-2-2l,1)=(2\cdot(\frac{n-2}{4})-2l-1,1)=(2\cdot(\frac{n-2}{4}-l)-1,1)\in S^*,$ for every  $l\in\{0,\dots,\frac{n-2}{4}-1\}$. Analogously, $u=(1,1)\in S^*$ and  $\Tilde{u}_l=(\mode{1+2l}_n,1)\in S^*,$ for every  $l\in\{1,\dots,\frac{n-2}{4}-1\}$.
Therefore, the result follows.
\end{proof}

We can summarize the results of this section as follows.

\begin{theorem}
    If $n\ge 3$ is an integer, then
    $$\xi(C_n\square K_2)=\left\{ \begin{array}{ll}
    \frac{5n-2}{4}; &  \text{ if }  \, \, n\equiv 2 \pmod 4,\\
    n; &  \text{ otherwise.}
    \end{array} \right. $$
\end{theorem}

\section{Squared grid graphs}

A \textit{grid graph} is known to be the Cartesian product graph of two paths $P_k$ and $P_n$. A grid $P_k\square P_n$ is \textit{squared} if $k=n$. In this section, we focus on finding the value of $\xi\left(P_n\square P_n\right)$ and shall use the following notation to refer to the vertices of $P_n\square P_n$.
$V(P_n\square P_n)=V(P_n)\times V(P_n)=\left\{1,...,n\right\}\times\left\{1,\dots,n\right\}=\left[n\right]\times\left[n\right]$.
Thus, if $v\in V(P_n\square P_n)$, then $v=(x_i,y_j)$ for some $i,j\in\left[n\right]$.

\begin{theorem}\label{th:grid}
Let $P_n$ be a path of order $n\ge 2$. Then
$$
\xi\left(P_n\square P_n\right)=\left\lceil\frac{n^2}{2}\right\rceil.
$$
\end{theorem}
\begin{proof}
We first show that $\xi(P_n\square P_n)\geq\left\lceil\frac{n^2}{2}\right\rceil$. Note that $P_n\square P_n$ is a bipartite graph, with bipartition sets $A_n=\{(x_i,y_j)\in V(P_n\square P_n):\mode{i}_2\neq \mode{j}_2\}$ and $B_n=\{(x_i,y_j)\in V(P_n\square P_n):\mode{i}_2=\mode{j}_2\}$, and that the cardinality of these sets are either equal (if $n$ is even), or differ by just one (if $n$ is odd). In the latter case, note that $|A_n|=|B_n|-1$.
Hence, if $S$ is a distance equalizer basis of $P_n\square P_n$, then by Proposition~\ref{PropoBipartiteUPC}, it holds that $A_n\subseteq S$ or $B_n\subseteq S$, which implies that $\xi(P_n\square P_n)=|S|\geq\min\{|A_n|,|B_n|\}=|A_n|$. Thus, if $n$ is even or $B_n\subseteq S$, then
$\xi(P_n\square P_n)\ge\left\lceil\frac{n^2}{2}\right\rceil$.
Now, assume that $n$ is odd and $A_n\subseteq S$. In this case, for the vertices $v=(1,1)$, $w=(n,n)$ and $s=(s_x,s_y)$ such that $d(v,s)=d(s,w)$, we deduce that $x_s+y_s=n+1\equiv 0 \pmod 2$, which implies $s\in B_n$, and so $\xi(P_n\square P_n)=|S|\ge |A_n|+1=\left\lfloor\frac{n^2}{2}\right\rfloor+1=\left\lceil\frac{n^2}{2}\right\rceil$.

\medskip
It remains to prove that $\xi(P_n\square P_n)\leq\lceil\frac{n^2}{2}\rceil$.
To this end, we claim that the set $B_n$, defined above, is a distance-equalizer set of $P_n\square P_n$.
Let $v=(i,j)$ and  $w=(k,r)$ be two vertices from $V(P_n\square P_n)\setminus B_n$ such that $i\le k$ and $j\le r$. Since $d(v,w)=k-i+r-j$ is even, there is at least one vertex $z$ that lies is a shortest path between $v$ and $w$, and such that $d(v,z)=d(z,w)$. If $d(v,w)/2$ is an odd number, then the vertex $z$ belongs to $B_n$, and we are done, i.e., $v,w$ are equidistant to $z\in B_n$.
Hence, we may assume that $d(v,w)/2=\frac{k-i+r-j}{2}$ is even, and so, $z\notin B_n$. We differentiate two cases.

\medskip
\noindent Case 1. $r-j\ge k-i$.
 If $k\le n-1$, then we can take $z=\left(k, r-\frac{k-i+r-j}{2}\right)$ and so there exists $z'=(k+1,y_z)\in B_n$, which is equidistant from $v$, $w$. Now, if $k=n$, then $i>1$. Thus, we
can take $z=\left(i, j+\frac{k-i+r-j}{2}\right)$ and so there exists $z'=(i-1,y_z)\in B_n$, which is equidistant from $v$, $w$.

\medskip
\noindent Case 2. $r-j\le k-i$. If $r\le n-1$, then we can take $z=\left( k-\frac{k-i+r-j}{2},r\right)$ and so there exists $z'=(x_z,r+1)\in B_n$, which is equidistant from $v$, $w$. Now, if $r=n$, then $j>1$. Thus, we
can take $z=\left( i+\frac{k-i+r-j}{2},j\right)$ and so there exists $z'=(x_z,j-1)\in B_n$, which is equidistant from $v$, $w$.

In order to complete the proof, we only need to consider the case $i\le k$ and $j\ge r$ in the choice of the vertices $v=(i,j)$ and  $w=(k,r)$. In such situation, the procedure is rather analogous to the one described above when $i\le k$ and $j\le r$. Thus, in every case, we confirm our claim that $B_n$ is a distance equalizer basis of $P_n\square P_n$, which completes the proof.
\end{proof}

We remark that the formula given in Theorem \ref{th:grid} seems to be true also in some other possible structures of (not squared) grid graphs, but not in every one. For example, in Table \ref{th:grid} some computer calculations are shown for several grids. One can see that the difference between the value of the formula from  Theorem \ref{th:grid} and the real value of the equidistant dimension of a grid $P_k\square P_n$ begins to increase as soon as the absolute value of the difference between $n$ and $k$ increases.

\begin{table}[h!]
$$
\begin{array}{c|ccccccccccccccccccc}
\tikz{\node[below left, inner sep=1pt] (k) {k};%
      \node[above right,inner sep=1pt] (n) {n};%
      \draw (k.north west|-n.north west) -- (k.south east-|n.south east);}
  & 2 & 3 & 4 & 5 & 6 & 7 & 8 & 9 & 10 & 11 & 12 & 13 & 14 & 15 & 16 & 17 & 18 & 19 & 20\\
\hline
2 & 0 & 0 & 1 & 1 & 1 & 1 & 2 & 2 & 3 & 3 & 3 & 4 & 4 & 4 & 5 & 5 & 6 & 6 & 7\\
3 & 0 & 0 & 0 & 0 & 0 & 0 & 1 & 1 & 2 & 2 & 2 & 2 & 2 & 2 & 4 & 4 & 4 & 4 & 5\\
4 & 1 & 0 & 0 & 0 & 1 & 1 & 2 & 2 & 2 & 2 & 3 & 3 & 3 & 3 & 4 & 4 & 5 & 5 & 6\\
5 & 1 & 0 & 0 & 0 & 0 & 0 & 0 & 0 & 1 & 1 & 2 & 2 & 3 & 3 & 3 & 3 & 3 & 3 & 4\\
6 & 1 & 0 & 1 & 0 & 0 & 0 & 1 & 1 & 2 & 2 & 3 & 3 & 3 & 3 & 4 & 4 & 4 & 4 & 5\\
7 & 1 & 0 & 1 & 0 & 0 & 0 & 0 & 0 & 0 & 0 & 1 & 1 & 2 & 2 & 3 & 3 & 4 & 4 & 4\\
8 & 2 & 1 & 2 & 0 & 1 & 0 & 0 & 0 & 1 & 1 & 2 & 2 & 3 & 3 & 4 & 4 & 4 & 4 & 5\\
9 & 2 & 1 & 2 & 0 & 1 & 0 & 0 & 0 & 0 & 0 & 0 & 0 & 1 & 1 & 2 & 2 & 3 & 3 & 4\\
10 & 3 & 2 & 2 & 1 & 2 & 0 & 1 & 0 & 0 & 0 & 1 & 1 & 2 & 2 & 3 & 3 & 4 & 4 & 5\\
11 & 3 & 2 & 2 & 1 & 2 & 0 & 1 & 0 & 0 & 0 & 0 & 0 & 0 & 0 & 1 & 1 & 2 & 2 & 3\\
12 & 3 & 2 & 3 & 2 & 3 & 1 & 2 & 0 & 1 & 0 & 0 & 0 & 1 & 1 & 2 & 2 & 3 & 3 & 4\\
13 & 4 & 2 & 3 & 2 & 3 & 1 & 2 & 0 & 1 & 0 & 0 & 0 & 0 & 0 & 0 & 0 & 1 & 1 & 2\\
14 & 4 & 2 & 3 & 3 & 3 & 2 & 3 & 1 & 2 & 0 & 1 & 0 & 0 & 0 & 1 & 1 & 2 & 2 & 3\\
15 & 4 & 2 & 3 & 3 & 3 & 2 & 3 & 1 & 2 & 0 & 1 & 0 & 0 & 0 & 0 & 0 & 0 & 0 & 1\\
16 & 5 & 4 & 4 & 3 & 4 & 3 & 4 & 2 & 3 & 1 & 2 & 0 & 1 & 0 & 0 & 0 & 1 & 1 & 2\\
17 & 5 & 4 & 4 & 3 & 4 & 3 & 4 & 2 & 3 & 1 & 2 & 0 & 1 & 0 & 0 & 0 & 0 & 0 & 0\\
18 & 6 & 4 & 5 & 3 & 4 & 4 & 4 & 3 & 4 & 2 & 3 & 1 & 2 & 0 & 1 & 0 & 0 & 0 & 1\\
19 & 6 & 4 & 5 & 3 & 4 & 4 & 4 & 3 & 4 & 2 & 3 & 1 & 2 & 0 & 1 & 0 & 0 & 0 & 0\\
20 & 7 & 5 & 6 & 4 & 5 & 4 & 5 & 4 & 5 & 3 & 4 & 2 & 3 & 1 & 2 & 0 & 1 & 0 & 0\\
\end{array}
$$
\caption{Table of values of $q=\xi(P_k\square P_n)-\left\lceil\frac{n\cdot k}{2}\right\rceil$}\label{tab:grids}
\end{table}

\section{Concluding remarks}

The results presented in this article allow to identify several research lines that can be further on developed. In particular, we remark the following ones.

\begin{itemize}
    \item Which is the value of $\xi(G)$ when $G$ is any $n$-dimensional Hamming graph for any $n\ge 3$?
    \item Which is the value of the $\xi(G)$ when $G$ is any grid graph (not squared)?

    Even more general, the case when $G$ is the Cartesian product of $k\ge 2$ paths seems to be very challenging and of interest.
    \item In view of the results from Section \ref{sec-hyper}, we wonder on which is the value of $\xi(Q_n)$ when $n\equiv0 \text{ mod }4$. We strongly believe that the exact value of it coincides with the upper bound of item Theorem \ref{th-hyper} (ii). That is, we suspect that $\xi(Q_n)=2^{n-1}+2^{\frac{n}{2}-2}$ when $n\equiv0 \text{ mod }4$.
    \item Study the equidistant dimension of Cartesian product graphs is general.
    \item Characterize the bipartite graphs for which the lower bound from Proposition \ref{PropoBipartiteUPC} is tight.
    \item Study the equidistant dimension of prism graphs is general.
\end{itemize}

\section*{Acknowledgement}
I. G. Yero has been partially supported by ``Ministerio de Ciencia, Innovaci\'on y Universidades'' through the grant PID2023-146643NB-I00.

\section*{Conflicts of interest}
The authors do not have any additional conflicts of interest to declare.


\begin{thebibliography}{99}

\bibitem{Caceres} J.~C\'aceres, C.~Hernando, M.~Mora, I.~M. Pelayo, M.~L. Puertas, C.~Seara, D.~R.~Wood, On the metric dimension of Cartesian products of graphs,
SIAM J.\ Discrete Math.\ 21(2) (2007) 423-441.

\bibitem{Erdos}
P.~Erd\"os, P.~Tur\'an, On some sequences of integers,
J.\ Lond.\ Math.\ Soc.\ 11 (1936) 261--264.

\bibitem{Gonzalez-2022}
A.~Gonz\'alez, C.~Hernando, M.~Mora, The equidistant dimension of graphs,
Bull.\ Malays.\ Math.\ Sci.\ Soc.\ 45 (2022) 1757--1775.

\bibitem{Gispert}
A.~Gispert-Fern\'andez, J.~A.~Rodr\'iguez-Vel\'azquez, The equidistant dimension of graphs: NP-completeness and the case of lexicographic product graphs,
AIMS Math.\ 9(6) (2024) 15325--15345.

\bibitem{Kratica}
J.~Kratica, M.~\v Cangalovi\'c, V.~Kova\v cevi\'c-Vuj\v ci\'c, Equidistant dimension of Johnson and Kneser graphs. arXiv:2406.17870 [math.CO] (25 Jun 2024)

\bibitem{Kuziak} D.~Kuziak, I.~G.~Yero,
Metric dimension related parameters in graphs: A survey on combinatorial, computational and applied results,
arXiv:2107.04877 [math.CO].

\bibitem{Savic}
A.~Savi\'c, Z,~Maksimovi\'c, M.~Bogdanovi\'c, J.~Kratica, The equidistant dimension of some graphs of convex polytopes. arXiv:2407.15307 [math.CO] (22 Jul 2024)

\bibitem{Tillquist}
R.~C.~Tillquist, R.~M.~Frongillo, M.~E.~Lladser,
Getting the lay of the land in discrete space: A survey of metric dimension and its applications,
SIAM Rev,\ 65(4) (2023) 919--962.

\bibitem{Trujillo-Rasua-2016}
R.~Trujillo{-}Rasua, I.G.~Yero,
$k$-metric antidimension: A privacy measure for social graphs,
Inf.\ Sci.\ 328 (2016) 403--417.

\end{thebibliography}
\end{document}